\documentclass[a4paper,11pt]{amsart}
\usepackage[latin1]{inputenc}

\usepackage{multicol}
\usepackage{amssymb, amsmath, amsthm}
\usepackage{latexsym}
\usepackage{dsfont}
\usepackage{graphicx}

\usepackage{curves} 
\usepackage{epic} 
\usepackage{eepic}
\usepackage{xcolor}
 
\input xy
\xyoption{all}

\usepackage{verbatim}

\usepackage{hyperref}

\usepackage{multicol}
\usepackage{amssymb, amsmath,amsthm}
\usepackage{latexsym}
\usepackage{dsfont}
\usepackage{graphicx}
\input xy
\xyoption{all}

\usepackage{verbatim}

\newtheorem{theorem}{Theorem}[section]
\newtheorem{corollary}[theorem]{Corollary}
\newtheorem{lemma}[theorem]{Lemma}
\newtheorem{proposition}[theorem]{Proposition}

\theoremstyle{definition}
       \newtheorem{definition}[theorem]{Definition}
       \newtheorem{remark}[theorem]{Remark}
       
       \newtheorem{parrafo}[theorem]{{$\!$}}

\newcommand{\nat}{\mathbb N}
 
\newcommand{\calo}{{\mathcal {O}}}

\newcommand{\codim}{{\rm codim}}

\newcommand{\ord}{{\rm ord}}

\newcommand{\Sing}{{\rm Sing}}
\newcommand{\Spec}{{\rm Spec}}

\newcommand{\R}{{\mathcal R}}

\newcommand{\G}{{\mathcal G}}
\newcommand{\C}{{\mathcal C}}
\renewcommand{\L}{{\mathcal L}}
\newcommand{\M}{{\mathfrak M}}

\renewcommand{\S}{{\mathbb S}}
\newcommand{\De}{\Delta^{(e)}}
\newcommand{\I}{{\mathbb I}}

\newcommand{\In}{\rm In}

\title{The $\tau$-invariant and elimination}

\author{Ang\'elica Benito}
\thanks{2000 {\em Mathematics subject classification. 14E15.}}
 \thanks{The author is partially supported by MTM2009-07291.}

\address{Dpto. Matem\'aticas,  Universidad
Aut\'onoma de Madrid and Instituto de Ciencias Matem\'aticas CSIC-UAM-UC3M-UCM \\
Ciudad Universitaria de Cantoblanco, 28049 Madrid, Spain}

\email[Ang\'elica Benito]{angelica.benito@uam.es}

\keywords{Singularities. Differential operators. Rees algebras.}

\begin{document}

\maketitle

\begin{abstract}
In this paper we present some results showing the good behavior of the $\tau$-invariant of a Rees algebra with integral closure and elimination (of variables).
\end{abstract}

\section{Introduction}
Hironaka's Theorem of embedded desingularization (\cite{Hir64}) was proven, over fields of characteristic zero, by induction on the dimension of the ambient space. It makes use of the existence of some special smooth hypersurfaces in the ambient space (\emph{hypersurfaces of maximal contact}). The problem of resolution of singularities is reformulated as an equivalent problem in these smooth hypersurfaces, that is, in a smooth scheme in one dimension less.
This form of induction holds exclusively over fields of characteristic zero, but fails over fields of positive characteristic. 

Results on resolution of singularities in positive characteristic in small dimension are due to Abhyankar. 
More recently some  strategies to deal with resolution of singularities over arbitrary fields have appeared in  works of  Kawanoue-Matsuki (\cite{Kaw}, \cite{KM}), Hironaka (\cite{Hironaka06}), W\l odarczyk (\cite{Wlo1}), Cossart-Piltant (\cite{CP1}, \cite{CP2}), Hauser (\cite{HaKang}), Villamayor (\cite{VV4}, \cite{VKyoto}) and Bravo-Villamayor (\cite{BV3}). 

These last three papers make use of a  form of induction in which hypersurfaces of maximal contact are replaced by generic projections on smooth schemes of lower dimension; and restrictions to these smooth hypersurfaces are replaced by a different form of elimination of variables. In this paper, we focus on these new form of induction.

When treating problems of embedded resolution of singularities, it is natural to identify  ideals with the same integral closure, for example when dealing with embedded principalization, also called Log-resolution of ideals, and also in Hironaka's notion of idealistic exponents. 

In this paper embedded ideals and idealistic exponents are reformulated as Rees algebras. The analog of two ideals with the same integral closure (as ideals) will be that of two Rees algebras with the same integral closure (as algebras).

In Theorem \ref{equal_tau} it is shown that the main invariant  considered in most inductive arguments (the Hironaka's $\tau$-invariant) it is compatible with integral closure of algebras (a similar result appears in Kawanoue's work,  \cite{Kaw}). 

The main result of this work, Theorem \ref{tauG_elim}, studies the behavior of Hironaka's $\tau$-invariant with the form of elimination mentioned above.

The $\tau$-invariant appears in \cite{Hir64}, and it is defined by geometrical conditions: given a hypersurface $X$ and a point closed $x\in X$, the $\tau$-invariant is the codimension of a linear subspace attached to the tangent cone of $X$ at $x$ (see \ref{vertices}). For an algebraic point of view, $\tau$ is   the minimum number of variables needed to express generators of $\mathbb{I}_X$, where $\mathbb{I}_X$ denotes the homogeneous ideal spanned by initial forms of $X$ at $x$. 

These an other aspects of $\tau$-invariant  are studied in works of Oda (\cite{Oda1987}, \cite{Oda1973} or \cite{Oda1983}), see also \cite{Hironaka70} and \cite{Kaw}.

The first sections of this paper presents the basic definitions of Rees algebras and its relationship with integral closure and differential operators (for more details see \cite{EncVil06} and \cite{VV1}); Elimination algebras are described in Section \ref{sect:elim:alg}; generic projections and formal definitions of tangent cones and $\tau$-invariant are introduced in Section \ref{sect:vertices}.

In Section \ref{sect:tau:int:clo} we prove that two integrally equivalent Rees algebras (i.e. with the same integral closure) have the same $\tau$-invariant. 

In Section \ref{sect:tau:elim} we show the main result of this paper, i.e. given a differential Rees algebra so that  $\tau_{\G,x}\geq 1$ locally at a closed point $x$ and if we are under conditions to consider its elimination algebra, say $\R_{\G}$, then the $\tau$-invariant drops by one, i.e. $\tau_{\R_{\G}}=\tau_{\G}-1$.

In \cite{VV4}, Proposition 5.12. it is proven that given a differential Rees algebra $\G$ so that  $\tau_{\G,x}>1$, then its elimination algebra, say $\R_{\G}$, is such that $\tau_{\R_\G}\geq 0$. Here, we give a more general result and generalized to any arbitrary characteristic $p\geq0$ a well-known property  in characteristic $0$.

\vspace{0.2cm}


\section{Rees algebras}
\begin{parrafo}
We start introducing our main tool, Rees algebras, and disscusing some nice properties which will be useful throughout this paper.

\end{parrafo}

\begin{definition} A \emph{Rees algebra} over a smooth scheme $V$ is  a locally finite generated subalgebra of $\calo_V[W]$, of the form
$$\mathcal{G}=\bigoplus_{k\geq 0}I_kW^k,$$
where $\{I_k\}_{k\geq0}$ is a sequence of sheafs of ideals, and $\G$ is locally a finitely generated $\calo_{V}$-algebra so that
\begin{enumerate}
\item $I_0=\calo_V$,
\item each $I_k$ is  such that $I_sI_t\subset I_{s+t} $ for $s,t\geq 0$.
\end{enumerate}
\end{definition}

 \begin{remark}
At any affine open subset $U(\subset V)$, there exists a finite set of elements
$$\mathcal{F}=\{f_{1}W^{n_1},\dots ,f_{s}W^{n_s}\},$$
 with $n_i\in\mathds{Z}_{\geq 1}$ and
$f_i\in \calo_V(U)$, so that the restriction to $U$ is of the form
$$\mathcal{G}(U)=\calo_V(U)[f_{1}W^{n_1},\dots ,f_{s}W^{n_s}] (\subset
\calo_V(U)[W]).$$
We say that $\mathcal{F}=\{f_{1}W^{n_1},\dots ,f_{s}W^{n_s}\}$ is a \emph{set of generators} of $\G$ locally at $U$. Any other element of $fW^n\in\G$ can be express by a weighted homogeneous polynomial of degree $n$ and coefficients in $\calo_{V}(U)$, say $F_n(Y_1,\dots,Y_n)$, where each variable $Y_j$ has weight $n_j$ and $f=F_n(f_{1},\dots,f_{s})$.
\end{remark}

\begin{remark}\label{rkK1} Rees algebras and Rees rings are closely related. We define a \emph{Rees ring} as a Rees algebra if in any affine open subset we can choose a set of generators of $\G$, say $\mathcal{F}=\{f_1W^{n_1},\dots ,f_sW^{n_s}\}$, such that all the weights $n_i$ are  $n_i=1$.
 
Rees algebras are integral closures of Rees rings: if $N$ is a positive integer divisible by
all the weights $n_i$, then
$$\calo_V(U)[f_1W^{n_1},\dots
,f_sW^{n_s}]=\bigoplus_{k\geq 0}I_kW^k\ (\subset \calo_V(U)[W]),$$ 
is integral over the Rees sub-ring $\calo_V(U)[I_NW^N](\subset
\calo_V(U)[W^N])$.
\end{remark}

\begin{definition}
Given two Rees algebras $\G_1=\bigoplus_{n\geq 0}I_n^{(1)}W^n$ and $\G_2=\bigoplus_{n\geq 0}I_n^{(2)}W^n$ we define a binary operation between them denoted by $\odot$.
 In any affine open subset $U$, if locally $\G_1=\calo_{V}(U)[f_1W^{n_1},\dots,f_rW^{n_r}]$ and $\G_2=\calo_{V}(U)[g_1W^{m_1},\dots,g_sW^{m_s}]$, then
$$\G_1\odot\G_2=\calo_{V}(U)[f_1W^{n_1},\dots,f_rW^{n_r},g_1W^{m_1},\dots,g_sW^{m_s}].$$
This algebra $\G_1\odot\G_2$ is the smallest algebra containing $\G_1$ and $\G_2$.
\end{definition}

\begin{definition}
A closed set is attached  to $\mathcal{G}$, called the \emph{singular locus of} $\G$,
$$\Sing(\mathcal{G}):=\{ x\in V\ |\ \nu_x(I_k)\geq k, \mbox{ for each }
k\geq 1\},$$ where $\nu_x(I_k)$ denotes the order of the ideal
$I_k$ at the local regular ring $\calo_{V,x}$.
\end{definition}

In the following Proposition we present an easier form to compute the singular locus in terms of the generators of $\G$.

\begin{proposition}\label{prop1} Given an affine open $U\subset V$, and
$\mathcal{F}=\{f_1W^{n_1},\dots ,f_sW^{n_s}\}$ the set of generators  as above, then
$$\Sing(\mathcal{G})\cap U= \bigcap_{1\leq i \leq s}\{ x\in U\ |\ \ord_x (f_i) \geq n_i\}.$$

\end{proposition}
\begin{proof}
 See \cite{VV1}, Proposition 4.4 (2).
 \end{proof}

\begin{parrafo}

We now introduce an equivalence relation between Rees algebras. This notion is closely related with the equivalence relation Hironaka defined in the ambient of couples $(J,b)$, where $J\subset\calo_{V}$ is a sheaf of ideals and $b$ is a positive integer. Here the pairs $(J,b)$ and $(J',b')$ are equivalent, say $(J,b)\sim (J',b')$, if and only if $J^{b'}$ and $(J')^{b}$ have the same integral closure.

\begin{definition}\label{equiv_Rees} Given two Rees algebras over $V$,  $\G_1$ and $\G_2$ we say that 
they are {\em integrally
equivalent} and denote it by $\G_1\sim\G_2$, if $\G_1$ and $\G_2$ have the same integral closure.
\end{definition}

\begin{proposition} The singular locus is compatible with integral closure. In other words: Let $\mathcal{G}_1$ and $\mathcal{G}_2$ be two integrally equivalent Rees algebras over $V$. Then:
$$\Sing(\mathcal{G}_1)=\Sing(\mathcal{G}_2).$$
\end{proposition}

\begin{proof}
See \cite{VV1} Proposition 4.4 (1).
\end{proof}

If $\mathcal{G}_1$ and $\mathcal{G}_2$ are  {\em integrally}
equivalent on $V$, the same holds for any open restriction, and
also for pull-backs by smooth morphisms $W\to V$.
\end{parrafo}

\begin{parrafo}\label{introdiff}{\bf Differential Rees algebras}.  Here we  introduce a particular class of Rees algebras, called the differential Rees algebras, and discuss  some significant properties. Some of these properties of differential Rees algebras have a very important meaning in the context of singularities as we will see.

Given a smooth scheme $V$ over a field $k$, there is a locally free sheaf over $V$, defined for each non-negative integer $s$, called the sheaf of differential operators of order $s$, and denoted by $Diff^s_k$, this sheaf is so that
\begin{enumerate}
\item for $s=0$, $Diff^0_k=\calo_V$, and 
\item for each $s\geq 0$ $Diff^s_k\subset Diff^{s+1}_k$.
\end{enumerate}

Given a sheaf of ideals $J\subset\calo_{V}$, we can define an \emph{extension of the sheaf of ideals} with the help of this differential operator. This extension of $J$, denoted by
$Diff^s_k(J)$, is so that over any affine open set $U$,
 $$Diff^s_k(J)(U)=\{D(f)\ |\ D\in Diff^s_k(U)\hbox{ and }f\in J(U)\},$$
 that is $Diff^s_k(J)(U)$ is obtained by adding to $J(U)$ the elements $D(f)$ for every $D\in Diff_k^s(U)$ and $f\in J(U)$.

Under this assumptions, it is easy to check that
\begin{enumerate}
\item $Diff_k^0(J)=J$, and 
\item For any $s\in\mathbb{Z}_{\geq 0}$, we have the following inclusion of sheaves of ideals in $\calo_V$, $$Diff_k^s(J)\subset Diff_k^{s+1}(J).$$
\end{enumerate}

Then, the definition of differential Rees algebras arises in a natural way, only by keeping track of the weight.

\begin{definition} \label{3def1} Given a smooth scheme $V$, we say that a Rees algebra $\G=\bigoplus_{n\geq 0} I_nW^n$,  is an \emph{absolute differential algebra} or simply a \emph{differential algebra}, if: 
\begin{itemize}
\item[(i)] for $n\geq 0$, $I_{n+1}\subset I_{n}$. 
\item[(ii)] For a suitable affine open covering of $V$, say$\{U_i\}$, and for every differential operator $D\in Diff_k^{r}(U_i)$, and $h\in I_n(U_i)$, then
 $$D(h)\in I_{n-r}(U_i) \hbox{ if } n\geq r.$$
 \end{itemize}

Condition (ii) can be stated as: 
\begin{itemize}
\item[(ii')] For each $n$ and each $0\leq r\leq n$, $Diff_k^{r}(I_n)\subset I_{n-r}$..
\end{itemize}
\end{definition}

\begin{parrafo}\label{ffbrd}{\bf Local description of differential operators.} Let  $V$ be a smooth scheme of dimension $n$ over a field $k$.  Given a closed point $x\in V$, consider a regular system of parameters
$\{x_1,\dots,x_n\}$ at  $\calo_{V,x}$. If $k'$ is the residue field (a finite extension of $k$),  the completion is defined by
$\widehat{\calo}_{V,x}=k'[[x_1,\dots , x_n]].$

At the completion level, we can define the Taylor morphism, a continuous $k'$-linear ring
homomorphism defined as:
$$
\xymatrix@R=0pc @C=1.5pc{
Tay:   k'[[x_1,\dots , x_n]]\ar[r]   & k'[[x_1,\dots , x_n,T_1,\dots ,
T_n]]\\
\qquad \ \ \  x_i \ar@{|->}[r]   & x_i+T_i
}
$$

For any
$f\in k'[[x_1,\dots , x_n]]$, 
$$Tay(f)=\sum_{\alpha \in
\mathbb{N}^n} g_{\alpha}T^{\alpha}, \hbox{ with } g_{\alpha} \in
k'[[x_1,\dots , x_n]].$$ 
Here, differential operators appear in the ambit of formal power series: Define, for each $\alpha \in
\mathbb{N}^n$, $\Delta^{\alpha}(f)=g_{\alpha}$. This $\Delta^\alpha$ are defined on the ring of formal power series, but considering the natural
inclusion of $\calo_{V,x}$ in its completion, it can be proved that 
$\Delta^{\alpha}(\calo_{V,x})\subset \calo_{V,x},$. Moreover the set 
$$\{ \Delta^{\alpha} \ |\  \alpha \in \nat^n ,\ 0\leq |\alpha| \leq r\}$$
 generates  $Diff^r_k$ locally at $x$ (for more details of these facts, see \cite{Gr} Theorem 16.11.2).

In the following Theorem it is formulated how differential Rees algebras are related to general Rees algebras.
\end{parrafo}

\begin{theorem}\label{thopG} Let $V$ be a smooth scheme. Given a Rees algebra $\mathcal{G}$ over 
$V$, there exists a differential Rees algebra denoted by $G(\mathcal{G})$, such that:
\begin{itemize}
\item[(i)] $\mathcal{G}\subset G(\mathcal{G})$.

\item[(ii)] If $\mathcal{G}\subset \widehat{\G}$ and $\widehat{\G}$ is a
differential algebra, then $G(\mathcal{G})\subset \widehat{\G}$.
\end{itemize}

Moreover, locally at $x\in V$,  a closed point, if
$\mathcal{F}=\{
f_{1}W^{n_1},\dots, f_{s}W^{n_s}\},$ is a local set of generators of $\G$, then
\begin{equation} 
\mathcal{F'}=\{\Delta^{\alpha}(f_{i})W^{n'_i-\alpha}\ |\
f_{i}W^{n_i}\in \mathcal{F}, \alpha \in \mathbb{N}^n,\ 0\leq |\alpha| < n'_i
\leq n_i, \hbox{ for }1\leq i\leq s\}
\end{equation}
is a set of generators of $G(\mathcal{G})$ locally at $x$.
\end{theorem}
\begin{proof} See \cite{VV1}, Theorem 3.4.
\end{proof}

\begin{remark}
The previous Theorem shows that there exists a smallest differential Rees algebra containing $\G$, it is denoted by $G(\G)$.
\end{remark}

\begin{remark}\label{rk33} By the local description of $\G$ and $G(\G)$, it is easy to check that
$$\Sing(\mathcal{G})=\Sing(G(\mathcal{G})).$$
\end{remark}
\end{parrafo}

\begin{parrafo}{\bf Relative differential algebras}. 
Let $V\overset{\phi}{\longrightarrow} V'$ be a smooth morphism of smooth schemes.
Denote by $Diff^r_{\phi}(V)$ the locally free sheaf of relative differential operators of order $r$.

\begin{definition} \label{3def1} A Rees algebra $\G=\bigoplus I_k W^k$
 over $V$  is a $\phi$-\emph{relative differential algebra}, if 
\begin{itemize}
\item[(i)] for $n\geq 0$, $I_{n+1}\subset I_{n}$. 
\item[(ii)] For a suitable affine open covering of  $V$, say $\{U_i\}$, and for every relative differential operator $D\in Diff_{\phi}^{r}(U_i)$, and $h\in I_n(U_i)$, 
 $$D(h)\in I_{n-r}(U_i) \hbox{ if } n\geq r.$$
 \end{itemize}

As in Definition \ref{3def1}, condition (ii) can be reformulated by: 
\begin{itemize}
\item[(ii')]  For each $n$, and $0\leq r
\leq n$, $Diff_\phi^{r}(I_n)\subset I_{n-r}$.
\end{itemize}
\end{definition}

\begin{remark}
Since $Diff^r_{\phi}(V)\subset Diff^r_{k}(V)$, then any differential algebra is also a 
$\phi$-relative differential algerba.
\end{remark}

\begin{remark}As in Theorem \ref{thopG}, any Rees algebra $\G$ can be extended to a smallest $\phi$-relative differential algebra. Given an ideal $J\subset \calo_V$ and a smooth morphism $V\overset{\phi}{\longrightarrow}
V'$ we can considered the natural  extension of ideals, say $J \subset Diff^r_{\phi}(J)$, defined for each open subset $U$ in $V$ as
$$Diff^r_{\phi}(J)(U)=\{ D(f)\ |\ f\in J(U) , D \in
Diff^r_{\phi}(U) \} $$ 
\end{remark}

Finally, a Rees algebra $\G=\bigoplus I_k W^k$
 is a $\phi$-relative differential algebra, if
 and only if $Diff^r_{\phi}(I_n)\subset I_{n-r},\hbox{  for any positive integers }r\leq n.$
\end{parrafo}

\section{Elimination algebras}\label{sect:elim:alg}

\begin{parrafo}\label{univ_elimalgebra}{\bf Universal elimination algebra.} Let $S$ be a ring. Consider the polynomial ring $S[Z]$ and a monic polynomial $f(Z)\in S[Z]$ such that $f(Z)= Z^n+a_1Z^{n-1}+\cdots+a_n$. As  in \cite{VV4}, elimination algebras arise quite naturally when we search for equations on the coefficients which are invariant by all change of variables, namely those of the form $Z\mapsto uZ+s$, where $\alpha\in S $ and $u\in U(S)$.  For this reason we discuss some aspects of invariant and elimination theory, and obtain some results in a universal way.

Let define $F_n(Z)=(Z-Y_1)(Z-Y_2)\dots(Z-Y_n)$ as the {\em universal monic polynomial of degree} $n$ in the polynomial ring of $n+1$ variables $k[Y_1,\dots,Y_n,Z]$. The group of permutations of $n$ elements, $\S_n$, acts on $k[Y_1,\dots,Y_n]$ by permuting the index of the variables $Y_1,\dots,Y_n$; and this action extends to $k[Y_1,\dots,Y_n,Z]$ by fixing $Z$.

The subring of invariants, say $k[Y_1,\dots,Y_n]^{\S_n}$, is generated, as a $k$-algebra, by the symmetric elemental functions of order $i$, $s_{n,i}$, for $1\leq i \leq n$:
\begin{align*}
&s_{n,1}= Y_1+\dots+Y_n\\
&s_{n,2}=\displaystyle\sum_{1\leq i<j\leq n} Y_iY_j\\
&\quad\vdots\\
&s_{n,n}=Y_1Y_2\dots Y_n
\end{align*}
That is, $k[Y_1,\dots,Y_n]^{\S_n}=k[s_{n,1},\dots,s_{n,n}]$ and
$$k[Y_1,\dots,Y_n,Z]^{\S_n}=k[s_{n,1},\dots,s_{n,n}][Z].$$

It is easy to check that the monic polynomial 
$$F_n(Z)=(Z-Y_1)\dots(Z-Y_n) \in k[s_{n,1},\dots,s_{n,n}][Z],$$ 
since this polynomial is invariant under the action of $\S_n$.  

Let $S$ be a $k$-algebra and fix $f(Z)= Z^n+a_1Z^{n-1}+\cdots + a_n$, a monic polynomial of degree $n$ in $S[Z]$.  This polynomial arises from the universal polynomial $F_n(Z)$ by the base change morphism 
$$
\xymatrix@R=0pc @C=1.5pc{k[s_{n,1},\dots,s_{n,n}] \ar[r] & S\\
s_{n,i} \ar@{|->}[r] &  (-1)^{i}\cdot a_i
}
$$
which induces a morphism
$$
k[s_{n,1},\dots,s_{n,n}][Z]\longrightarrow S[Z]
$$
that maps $F_n(Z)$ to $f(Z)$.

So results in  the  universal setting will provide results for the fixed polynomial $f(Z)$.

The group $\S_n$ acts linearly over the polynomial ring 
$k[Y_1,\dots,Y_n,Z]$ and this action preserves the grading of the ring. So the invariant subring given by $k[s_{n,1},\dots,s_{n,n}][Z]$ can be considered as a graded sub-ring (with the grading inherited from that of $k[Y_1,\dots,Y_n,Z]$). The group  $\S_n$ also acts linearly in the graded sub-ring $k[Y_i-Y_j]_{1\leq i,j,\leq n}\subset k[Y_1,\dots,Y_n]$ defining an inclusion of graded sub-rings 
$$k[Y_i-Y_j]^{\S_n}_{1\leq i,j,\leq n}\subset k[Y_1,\dots,Y_n]^{\S_n}.$$ 
Set $k[Y_i-Y_j]^{\S_n}=k[H_{n_1},\dots,H_{n_r}]$, where the $H_{n_i}$ are homogeneous polynomial of degree $n_i$, $1\leq i\leq r$. Note that $H_i$ is also weighted homogeneous of degree $n_i$ in $k[s_{n,1},\dots,s_{n,n}]$, where each $s_i$ has degree $i$, that is,
$$H_{n_i}=H_{n_i}(s_{n,1},\dots,s_{n,n}).$$

The graded algebra  $k[Y_i-Y_j]^{\S_n}$ will be called the \emph{universal elimination algebra}, any polynomial in this ring provides, for each
$f(Z)= Z^n+a_1Z^{n-1}+\cdots + a_n$ in $S[Z]$, and each base change as above, a function on the coefficients $a_i$ which is invariant by changes of variable of the form $Z \mapsto Z-s$, $s\in S$. In other words, we have obtain functions in the coefficients
$$h_{n_i}(a_1,\dots,a_{n})$$
which are invariants by changes of the form $Z \mapsto Z-s$.

Note that if we consider now more general changes of variables, namely these of the form $Z\mapsto uZ-s$, $s\in S$ and $u\in U(S)$, we obtain that the previous functions are of the form
$$u^{n_i}h_{n_i}(a_1,\dots,a_{n}).$$

Nevertheless, we are considering graded algebras, and in particular, for every weight $n\in\mathbb{Z}_{\geq 0}$ we will consider an ideal $J_n$ generated by weighted homogeneous polynomials $h$ of degree $n$ and coefficients in $S$, say $H_n(V_1,\dots,V_n)$, where each variable $V_j$ has weight $n_j$ and $h=H_n(h_{n_1},\dots,h_{n_r})$. So, every ideal $J_n$ of the elimination algebra is invariant by changes of the form $Z\mapsto uZ+s$ even if the functions $h_{n_i}$ are not invariant.
It is enough to consider the particular case of changes of the form $Z\mapsto Z+s$.

\end{parrafo}
\begin{parrafo}\label{tayloruniv}{\bf The Taylor morphism in the universal setting.}
A morphism $Tay$ is defined as in (\ref{ffbrd}). Let $S$ be a $k$-algebra and consider the $S$-homomorphism
$$
\xymatrix@R=0pc @C=1.5pc{
Tay:   S[Z]\ar[r]   & S[Z,T]\\
\ \ \ \  Z \ar@{|->}[r]   & Z+T
}
$$
For $f(Z)\in S[Z]$, we have 
$$Tay(f(Z))=\sum_{\alpha\in\mathbb N}g_\alpha(Z)T^\alpha,$$
with $g_\alpha(Z)\in S[Z]$ and finally define, for each $\alpha\in\mathbb N$, $\Delta^{(\alpha)}(f(Z))=g_\alpha(Z)$. 
\end{parrafo}

\begin{remark}\label{invasdif}
Since $F_n(Z)=(Z-Y_1)\cdot (Z-Y_2)\cdots (Z-Y_n)\in k[Y_1,\dots,Y_n][Z]$, then 
$$F_n(T+Z)=(T+(Z-Y_1))\cdot (T+(Z-Y_2))\cdots (T+(Z-Y_n)).$$
The coefficients of this polynomial in the variable $T$ are the symmetric polynomials evaluated on
the element $(Z-Y_1,\dots,Z-Y_n)$. So
\begin{equation}\label{eqinvasdif}
\Delta^{(e)}(F_n(Z))=(-1)^{n-e}s_{n,n-e}(Z-Y_1, Z-Y_2,\dots , Z-Y_n).
\end{equation}
Here  $ \mathbb{S}_n$ acts on the graded sub-ring $k[Z-Y_1,\dots Z-Y_n](\subset k[Y_1, \dots, Y_n,Z])$ setting $\sigma(Z)=Z$ for every $\sigma\in\S_n$ and preserving the graded structure. Note that
$$k[Z-Y_1,\dots , Z-Y_n]^{ \mathbb{S}_n}=k[F_n(Z),\Delta^{(e)}(F_n(Z))]_{ e=1,\dots, n-1},$$
and that each $\Delta^{(e)}(F_n((Z))$ is homogeneous of degree $n-e$.\end{remark}

As $Y_i-Y_j=(Z-Y_j)-(Z-Y_i)$ we deduce, from the inclusion 
$$k[Y_i-Y_j]_{1\leq i,j,\leq n} \subset k[Z-Y_1,\dots Z-Y_n],$$
an inclusion of graded sub-rings 
\begin{equation}\label{elimuniv}
k[H_{n_1},\dots,H_{n_r}]=k[Y_i-Y_j]^{\S_n}\subset k[F_n(Z),\Delta^{(e)}(F_n(Z))]_{ e=1,\dots, n-1},
\end{equation}
so now, each $H_{n_i}$ is weighted homogeneous in $k[F_n(Z),\Delta^{(e)}(F_n(Z))]_{e=1,\dots, n-1}$.
\begin{parrafo}\label{specelimalgebra}{\bf Specialization of the elimination algebra.}
We will now assign, to the monic polynomial $f(Z)= Z^n+a_1Z^{n-1}+\cdots+a_n$ of degree $n$ in $S[Z]$, a Rees algebra which is a subalgebra of the ring $S[Z][W]$ (i.e. finitely generated sub-algebra of $S[Z][W]$ for a dummy variable $W$). 

To be precise, we attach to a graded subring in $k[s_{n,1},\dots,s_{n,n}][Z]$ a subring in $S[Z][W]$, so that whenever $H$ is a weighted homogeneous polynomial of degree, say $m$, in $k[s_{n,1},\dots,s_{n,n}][Z]$, we assign to it an element of the form $hW^m$, with $h\in S[Z]$.

Given $f(Z)= Z^n+a_1Z^{n-1}+\cdots+a_n$  in $S[Z]$ we define a $k$-algebra homomorphism on $S[Z][W]$ by setting
$$
\xymatrix@R=0pc @C=1.5pc{k[s_{n,1},\dots,s_{n,n}] [Z]\ar[r] & S[Z][W]\\
s_{n,i} \ar@{|->}[r] &  (-1)^{i}\cdot a_iW^i\\
Z \ar@{|->}[r] &  ZW.
}
$$
Any graded sub-ring in $k[s_{n,1},\dots,s_{n,n}] [Z]$ defines now a graded sub-algebra in $S[Z][W]$, and from (\ref{elimuniv})  we obtain 
\begin{equation}\label{elimuniv1}
S[h_{n_i}W^{n_i}]\subset S[f(Z)W^n,\Delta^{(e)}(f(Z))W^{n-e}]_{1\leq e\leq n-1}
\end{equation}

Note that $k[Y_i-Y_j]^{\S_n}$ does not involve the variable $Z$, so that $S[h_{n_i}W^{n_i}]\subset S[W]$. For this reason, the algebra $S[h_{n_i}W^{n_i}]$ is called the \emph{elimination algebra}.

\begin{remark}
The elements $\Delta^{(e)}(f(Z))W^{n-e}$ in (\ref{elimuniv1}) are exactly the relatives differential operators applied to $f(Z)$ with weight $n-e$. Consider the Taylor morphism 
$$
\xymatrix@R=0pc @C=1.5pc{
Tay:k[s_{n,1},\dots,s_{n,n}][Z]\ar[r] & k[s_{n,1},\dots,s_{n,n}][Z,T]\\
\ \ \ \  Z \ar@{|->}[r] &  Z+T\\
}
$$
 and the base change morphism $k[s_{n,1},\dots,s_{n,n}]\longrightarrow S$ defined as above. Because of the good behaviour of differentials with base change and by (\ref{tayloruniv}), we obtain $\Delta^{(e)}(f(Z))W^{n-e}$  from the $\Delta^{(e)}(F_n(Z))W^{n-e}$. 
\end{remark}
\end{parrafo}

\begin{parrafo}\label{elim2poly}{\bf Elimination algebra of two polynomials}.
A universal elimination algebra was defined in (\ref{univ_elimalgebra}) for one universal polynomial. These ideas have a natural extension to the case of several polynomials. Here we only consider the case of two polynomials, but the arguments are similar for the case  of more than two.

Fix two positive integers, $r$, $s\in \mathbb{N}$ such that $r+s=n$ and consider $F_r(Z)=(Z-Y_1)\dots(Z-Y_r)$ and $F_s(Z)=(Z-Y_{r+1})\dots(Z-Y_{n})$. The permutation group $\S_r$ acts on $k[Y_1,\dots,Y_r]$ and $\S_s$ acts on $k[Y_{r+1},\dots,Y_n]$. Define 
$$k[H'_{m_1},\dots,H'_{m_{r,s}}]:=k[Y_i-Y_j]_{1\leq i,j\leq n}^{\S_r\times\S_s}$$
as \emph{the universal elimination algebra for two polynomials}. Since $\S_r\times\S_s\subset\S_n$,  there is a natural inclusion
$$k[H_{m_1},\ldots,H_{m_n}] :=k[Y_i-Y_j]_{1\leq i,j\leq
n}^{\S_n}\subset k[H'_{m_1},\ldots,H'_{m_{r,s}}], $$
which is a finite extension of graded algebras.
On the other hand, one can check that
\begin{multline*}
k[Z-Y_1,\ldots,Z-Y_n]^{\S_r\times \S_s} \\
=k[F_r(Z),\Delta^{(e)}(F_r(Z)), F_s(Z),\Delta^{(\ell)}(F_s(Z))]_{e=1,\ldots,r-1,\ell=1,\ldots,s-1}.
\end{multline*}

The inclusion of finite groups $\S_r\times\S_s\subset\S_n$ also shows that there is an inclusion of invariant algebras:
\begin{multline*}
 k[F_n(Z),\Delta^{(j)}(F_n(Z))]_{j=1,\ldots,n-1} \\
  \subset k[F_r(Z),\Delta^{(e)}(F_r(Z)), F_s(Z),\Delta^{(\ell)}(F_s(Z))]_{e=1,\ldots,r-1,\ell=1,\ldots,s-1}.
 \end{multline*}
which is a finite extension of graded rings.

Note that  $\Delta^{(e)}(F_r(Z))$ is homogeneous of degree
$r-e$ for $e=1,\ldots,n-1$, and that $k[Z-Y_1,\ldots,Z-Y_n]^{\S_r\times \S_s}$ is a graded 
subring in $k[Y_1,\ldots,Y_,Z]$. Moreover, there is a natural inclusion
\begin{equation}\label{eqrp}
k[H'_{m_1},\ldots,H'_{m_{r,s}}] \subset k[F_r(Z),\Delta^{(e)}(F_r(Z)), F_s(Z),\Delta^{(\ell)}(F_s(Z))]_{e=1,\ldots,r-1,\ell=1,\ldots,s-1}
\end{equation}
that arises from $k[Y_i-Y_j]_{1\leq i,j \leq n}\subset k[Z-Y_1,\ldots,Z-Y_n].$

Here $F_r(Z)F_s(Z)=F_n(Z) $. The rings
$k[Y_1,\ldots,Y_r]^{\S_r}=k[v_1,\ldots,v_r]$, and $k[Y_{r+1},\ldots,Y_n]^{\S_s}=k[w_1,\ldots,w_s]$ are graded subrings in $k[Y_{1},\ldots,Y_n]$, $F_r(Z)\in k[v_1,\ldots,v_r][Z]$, $F_s(Z)\in k[w_1,\ldots,w_r][Z]$, and there is an inclusion
$$k[H'_{m_1},\ldots,H'_{m_{r,s}}] \subset k[v_1,\ldots,v_r, w_1,\dots ,w_s]$$
arising from  $k[Y_i-Y_j]\subset k[Y_1, \dots, Y_n]$. In particular each $H'_{j}$ is also a weighted homogeneous polynomial in the {\em universal coefficients} $\{v_1,\ldots,v_r, w_1,\ldots,w_s\}$.
\end{parrafo}

\begin{parrafo}\label{arrastre}{\bf Specialization for two polynomials.}
The previous discussion, for the case of two polynomials extends to the case of several polynomials.
These algebras  specialize into subalgebra of the form 
\begin{equation}\label{eqara1}
S[Z][f_i(Z)W^{n_i},\Delta^{(e_i)}(f_i(Z))W^{n_i-e_i}]_{1\leq e_i\leq n_i-1,\ 1\leq i\leq r},
\end{equation}
in the sense of (\ref{specelimalgebra}),
where $f_i(Z)$ are monic polynomials of the form $$f_i(Z)=Z^{n_i}+a_{1,i}Z^{n_i-1}+\dots+a_{n_i,i}$$ for $i=1,\dots,r$. 
The same specialization, applied to the universal elimination algebras (free of the variable $Z$), define algebras, say 
\begin{equation}\label{eqara2}
S[h^{(j)}_{n_1,\dots , n_r}W^{N_{n_1,\dots , n_r}^{(j)}}]_{1\leq j \leq R_{n_1,\dots , n_r}}\subset S[W],
\end{equation}
for suitable positive integers $R_{n_1,\dots , n_r}$.

An important property of  specializations in  (\ref{specelimalgebra}) is their compatibility with finite extensions of graded algebras. So, for example, in the case two polynomial discussed in (\ref{elim2poly}),
we conclude that if $f_n(Z)\in S[Z]$ is a monic polynomial of degree $n$ which factorizes as a product of monic polynomials, say $f_n(Z)=f_r(Z)f_s(Z)$, then there is a natural (and finite!) inclusion of graded rings: 
\begin{multline*}
S[Z][f_n(Z)W^{n},\Delta^{(j)}(f_n(Z))W^{n-j}]_{ 1\leq j\leq n}\\ 
\subset S[Z][f_r(Z) W^r,\Delta^{(e)}(f_r(Z))W^{r-e}, f_s(Z),\Delta^{(\ell)}(f_s(Z))W^{s-\ell}]_{e=1,\ldots,r-1,\ \ell=1,\ldots,s-1}
\end{multline*}
(as subalgebras of $S[Z][W]$). Similarly, a finite extension of graded subalgebras of $S[W]$ is defined by the specialization of the corresponding elimination algebras.
The same holds for more than two polynomials.

\end{parrafo}

\section{Linear space of vertices}\label{sect:vertices}

\begin{parrafo}\label{genproj}
 Let $V^{(d)}$ denote a smooth scheme of dimension $d$, and let $X\subset V^{(d)}$ be a hypersurface such that $X=V(\langle f\rangle)$ locally at a $n$-fold point 
$x\in V^{(d)}$. So $n=\max-\ord f$, the maximum order of the hypersurface in a neighborhood of $x$. We claim that for a sufficiently {\em generic} projection 
$$
\xymatrix{ V^{(d)} \ar[r]^\beta  &
                      V^{(d-1)} }
$$
the hypersurface $X$ can be express, in \'etale topology, as $X=V(f(Z))$, where $f(Z)\in \calo_{V^{(d-1)}, \beta(x)}[Z]$ is a monic polynomial of degree $n$ in a variable $Z$. This will hold under a suitable geometric condition, that will be expressed on $\mathbb{T}_{V^{(d)} ,x}$, the tangent space at the point. In fact, we will show that such conditions on $f$ can be achieved whenever the tangent line, at $x$, of the smooth curve $\beta^{-1}(\beta (x))$, say $\ell\subset \mathbb{T}_{V^{(d)} ,x}$ and the tangent cone of the hypersurface at the point, say $\C_f\subset \mathbb{T}_{V^{(d)} ,x}$, are in general position (intersect only at the origin). 

Let $\{x_{1},\dots,x_{d}\}$ be a regular system of parameters at $\calo_{V^{(d)} ,x}$. Recall that 
$$ \mathbb{T}_{V^{(d)} ,x}=\Spec(gr_{\M}(\calo_{V^{(d)} ,x})),$$
$$gr_{\M}(\calo_{V^{(d)} ,x})=k\oplus\M/ \M^2\oplus\M^2/\M^3\oplus\dots\oplus\M^r/\M^{r+1}\oplus\dots \cong k[X_{1},\dots,X_{d}] ,$$
where $X_i$ denotes the class of $x_i$ in $\M/\M^2$, that is, the initial form.
Let us compute now the initial form of $f$, to do so consider the following exact sequences:
$$\begin{array}{ccccccc}
0\longrightarrow & \langle f\rangle & \longrightarrow & \calo_{V^{(d)} ,x}& \longrightarrow & \calo_{X ,x} & \longrightarrow 0\\
0 \longrightarrow & \M^r\cap \langle f\rangle  & \longrightarrow & \M^r & \longrightarrow & \overline{\M}^r &\longrightarrow 0\\
0 \longrightarrow & [\In(\langle f\rangle)]_{r}  & \longrightarrow & \M^r/\M^{r+1} & \longrightarrow & \overline{\M}^r/ \overline{\M}^{r+1} &\longrightarrow 0
\end{array}$$
where $[\In(\langle f\rangle)]_{r}$ denotes the ideal of the homogeneous forms of degree $r$ in the homogeneous ideal $\In(\langle f\rangle)$.

Here $\M^r/ \M^{r+1}=\overline{\M}^{r}/ \overline{\M}^{r+1}$  for every $r<n$, and the first time that equality fails to hold is at $r=n$; that is, in degree $n$, where the  initial form of $f$, say $\In(f)$, appears. So $gr_{\M}(\calo_{X ,x})=k[X_{1},\dots,X_{d}]/\langle\In(f)\rangle,$ and the tangent cone of $X$ at $x$ is 
$$\C_f=\Spec(gr_{\M}(\calo_{X ,x}))=\Spec(k[X_{1},\dots,X_{d}]/\langle\In(f)\rangle) (\subset \mathbb{T}_{V^{(d)} ,x}).$$
\end{parrafo}

\begin{parrafo}\label{proptransv}
Fix now a smooth morphism $V^{(d)} \overset{\beta}{\longrightarrow} V^{(d-1)} $, defined at a neighborhood of $x$, and let  $\ell$ denote the smooth curve $\beta^{-1} (\beta(x))$. 

As  $f$ has multiplicity $n$ at $\calo_{V^{(d)} ,x}$, the class of $f$, say $\overline{f}$, has order at least $n$ at the local regular ring $\calo_{\ell ,x}$. Moreover, the order is $n$ if and only if the tangent line of $\ell$ at $x$ and $\C_f$ intersect only at the origin of 
$ \mathbb{T}_{V^{(d)} ,x}$.

If we fix a regular system of parameters, say $\{x_1,\dots,x_{d-1}\}$, at $\calo_{V^{(d-1)} ,\beta(x)}$, the ideal defining $\ell$ is given by $\langle x_1,\dots,x_{d-1}\rangle$, and a new parameter $Z$ can be added so as to define a regular system of parameters at $\calo_{V^{(d)} ,x}$. 

Note finally that the geometric condition imposed at the point $x$ can also be expressed by $\Delta^{(n)}(f)(x)\neq 0$, where $\Delta^{(n)}$ is a suitable differential operator of order $n$, relative to $\beta: V^{(d)} \to V^{(d-1)} $. The advantage of this new reformulation in terms of differential operators is that it shows that if the geometric condition holds, for $X$ and $\beta$ at $x$, it also holds at any $n$-fold point of $X$ in a neighborhood of $x$.

Consider now the completion of the local rings in the  previous projection, that is, $\widehat\calo_{V^{(d)} ,x}$ and $\widehat\calo_{V^{(d-1)} ,\beta(x)}$, and apply Weierstrass Preparation Theorem, so the polynomial $f$ can be expressed as
$$u\cdot f(x_{1},\dots,x_{d-1},Z)=Z^n+a_{1}Z^{n-1}+\dots+a_{n},$$
where $u$ is a unit,  $\{x_1,\dots,x_{d-1}\}$ is a regular system of parameters at $\widehat\calo_{V^{(d-1)} ,\beta(x)}$, and after adding the variable $Z$, $\{x_1,\dots,x_{d-1},Z\}$ is a regular system of parameters at  $\widehat\calo_{V^{(d)} ,x}$,  and $a_i\in\widehat\calo_{V^{(d-1)} ,\beta(x)}$.

An similar result holds replacing completion by henselization. In this case, the coefficients $a_i$ are functions in a \'etale neighbourhood of the point $\beta(x)$ in $V^{(d-1)}$.
\end{parrafo}

\begin{parrafo}{\bf The linear space of vertices.}\label{vertices}
Let $V^{(d)}$ be a smooth scheme, $X$ a hypersurface locally described by $f$, and $\C_f\subset \mathbb{T}_{V^{(d)} ,x}$ the tangent cone associated to $X$ at $x$. Given a vector space $\mathbb{V}$, a vector $v\in \mathbb{V}$ defines a translation, say $tr_v(w)=w+v$ for $w\in\mathbb{V}$. There is a largest linear subspace, say $\L_f$, so that $tr_v(\C_f)=\C_f$ for any $v\in\L_f$, this subspace is called the \emph{ linear space of vertices}.

An important property of this subspace $\L_f$ is that for any smooth center $Y$ in $X$, such that $x\in Y$ and $X$ has multiplicity $n$ along $Y$,  the tangent space of $Y$, say $\mathbb{T}_{Y,x}$, is a subspace of $\L_f$.

There is a characterization of this linear space in algebraic terms. An homogeneous ideal $I$ is said to be \emph{closed by differential operators} if for any  homogeneous element $g\in I$ of order $n$ and any differential operator $\Delta^\alpha$ of order $|\alpha|\leq n-1$, then $\Delta^\alpha(g)\in I$. 

\begin{proposition}
Let $I=\langle f_1,\dots,f_r\rangle$ be an homogeneous ideal, where each $f_i$ is an homogeneous element of order $n_i$ for $i=1,\dots,r$. There exists a smallest extension of $I$ to an ideal closed by differential operators, say $\widetilde{I}$, given by
$$\widetilde{I}=\langle\Delta^{\alpha_1}(f_1),\dots,\Delta^{\alpha_r}(f_r)\rangle_{1\leq|\alpha_i|\leq n_i-1}.$$
\end{proposition}

\begin{proof}
Any homogeneous element of $I$ can be expressed as a homogeneous combination of products of the generators. So it is enough to consider the case $g=f_i\cdot f_j$, where $i,j\in\{1,\dots,r\}$. We claim that $\Delta^{\alpha}(g)\in\widetilde{I}$ for  every $\alpha$ so that $|\alpha|<n_i+n_j$.

Applying the product rule to $g$, we obtain 
$$\Delta^\alpha(f_i\cdot f_j)=\sum_{\alpha_1+\alpha_2=\alpha}\Delta^{\alpha_1}(f_i)\Delta^{\alpha_2}(f_j). $$
Since $\Delta^{\alpha_1}(f_i)=0=\Delta^{\alpha_2}(f_j)$ for $\alpha_1>n_i$, $\alpha_2>n_j$ and $\Delta^{n_i}(f_i)=\Delta^{n_j}(f_j)=1$, we deduce that $\Delta^{\alpha}(g)$ is a linear combination of elements of $\widetilde{I}$.
\end{proof}

\begin{remark}\label{rmk:initial_forms}
Given a regular system of parameters $\{x_1,\dots,x_d\}$ at $\calo_{V^{(d)},x}$, any homogeneous ideal closed by differential operators is defined by
\begin{itemize}
\item Linear forms,
\item Elements of $k[X_1^p,\dots,X_d^p]$,
\item $\dots$
\item Elements of $k[X_1^{
p^m},\dots,X_d^{p^m}]$ for some positive integer $m\in\mathbb{Z}_{>0}$.
\end{itemize}
Suppose now that $k$ is a perfect field, then any homogeneous ideal closed by differential operators $\widetilde{I}$ is, after linear change of variables, of the form
$$\widetilde{I}=\langle X_1,\dots,X_{\tau_0},X_{\tau_0+1}^p\dots,X_{\tau_1}^p,\dots,X_{\tau_{m-1}+1}^{p^m},\dots,X_{\tau_m}^{p^m}\rangle.$$
\end{remark}

If we extend $\langle In(f)\rangle$ to the smallest ideal closed by differential operators, say $\widetilde{I}$, then the zero-set of this homogeneous ideal defines the subspace $\L_f$ we have just defined. In these arguments we are assuming that the underlying field $k$ is perfect.

Similar notions can be defined for Rees algebras. Let $\G$ be a Rees algebra on the smooth scheme $V^{(d)}$. An homogeneous ideal is defined by $\mathcal{G}$ at $x$, say $In_x(\mathcal{G})$, included in $gr_{\M_x}( {\calo}_{V^{(d)},x})$; namely that homogeneous ideal generated by the class of $I_s$ at the quotient $\M_x^s/\M_x^{s+1}$, for all $s$.  Denote this ideal by $\I_{\G,x}$. The ideal $\I_{\G,x}$ defines a cone, say $\C_{\mathcal{G}}$, at $\mathbb{T}_{V^{(d)},x}$. Recall that there is a largest  subspace, say  $\L_{\mathcal{G}}$,  included and acting by translations on $\C_{\mathcal{G}}$. 

One can check that $In_x(G(\mathcal{G}))$  is the smallest (homogeneous) extension of  $\I_{\G,x}=In_x(\mathcal{G})$, closed by the action of the differential operators $\Delta^\alpha$; that is, with the property that if $F$ is an homogeneous polynomial of degree $N$ in the ideal, and if $|\alpha|\leq N-1$, then also $\Delta^\alpha(F)$ is in the  ideal. This homogeneous ideal defines the subspaces $\L_{\mathcal{G}}$,  included in $\C_{\mathcal{G}}$, with the properties stated before.

Recall that $\Sing(\mathcal{G})= \Sing(G(\mathcal{G}))$ (see (\ref{rk33})). The previous discussion shows how the homogeneous ideal at $x$ attached to $G(\G)$, say $In_x(G(\G))$, relates to the one attached to $\G$, say $In_x(\G)$: If $\C_\G$ is the cone associated with $\G$, then the cone associated to $G(\G)$ is the linear subspace $\L_\G$.
\end{parrafo}

\begin{definition}({\bf Hironaka's $\tau$-invariant}). 
$\tau_{{\mathcal G},x}$
will denote the minimum number of variables required to express generators of the ideal
$\I_{\G,x}$. This algebraic definition can be reformulated geometrically: $\tau_{\G,x}$ is the codimension of the linear subspace $\L_{\G,x}$ in $\mathbb{T}_{V^{(d)},x}$.
\end{definition}

\section{The $\tau$-invariant and integral closure of Rees algebras}\label{sect:tau:int:clo}

\begin{parrafo}
In this section, we give an easy proof of the compatibility of $\tau$-invariants and integral equivalence, i.e. given two integrally equivalent Rees algebras, they have the same $\tau$-invariant. This assertion will be needed in the proof of Theorem \ref{tauG_elim}. The proof is based on properties of integrally equivalent algebras and the algebraic definition of the $\tau$-invariant. For alternative proofs see, for example, \cite{Kaw}.
\end{parrafo}

\begin{theorem}\label{equal_tau}
If $\G$ and $\G'$ are two Rees algebras over $V$ with the same integral closure (i.e. $\G$ and $\G'$ are integrally equivalent), then for each $x\in\Sing \G=\Sing \G'$, there is an equality between their $\tau$-invariants, that is, $\tau_{\G,x}=\tau_{\G',x}$.
\end{theorem}

Some auxiliary results will be needed to prove the previous Theorem:

\begin{lemma}\label{lemma_aux1}
Given a Rees algebra $\G=\oplus I_nW^n$ defined locally at $x$ by the set of generators $\{f_{1}W^{n_1},\dots,f_{s}W^{n_s}\}$, i.e. $\G=\calo_{V,x}[f_{1}W^{n_1},\dots,f_{s}W^{n_s}]$, then 
$$\I_{\G,x}=\langle In_{n_1}(f_{1}),\dots,In_{n_s}(f_{s})\rangle.$$
\end{lemma}

\begin{proof}
Take $h_nW^n\in I_nW^n$. There exits  a weighted homogeneous polynomial of degree $n$, say $G_n(Y_1,\dots,Y_s)\in\calo_V[Y_1,\dots,Y_n]$, where each $Y_i$ has weight $n_i$, such that $G_n(f_{1}W^{n_1},\dots,f_{s}W^{n_s})=h_nW^n$. Considering the initial form of this last expression it follows that $In_{n_i}(f_{i})$ generate every initial form of $In_n(I_n)$ for any $n$, therefore the equality holds.
\end{proof}

\begin{lemma}\label{lemma_aux2}
Considered the Rees algebra defined locally at $x$ as 
$$\G=\calo_{V,x}[f_{1}W^{n_1},\dots,f_{s}W^{n_s}].$$ Let  $N>0$ be an integer such that $N$ is a common multiple of every $n_i$,  $i=1,\dots,s$. Define by $\G_N=\calo_V[I_NW^N]$ the Rees ring attached to $I_N$. Then,
\begin{itemize}
\item[$1.$] $\I_{\G_N,x}=\langle In_N(I_N)\rangle$.
\item[$2.$] $\sqrt{\I_{\G_N,x}}=\sqrt{\langle In_N(I_N)\rangle}=\sqrt{\I_{\G,x}}\ .$
\end{itemize}
\end{lemma}

\begin{proof}
\begin{itemize}
\item[$1.$] $\I_{\G_N,x}=\langle In_{mN}(I_{mN})\rangle_{m\geq0}$, and note that 
$$\langle In_{mN}(I_{mN})\rangle=\langle In_N(I_N)\rangle^m\subset \langle In_N(I_N)\rangle.$$
\item[$2.$] Let us check $\sqrt{\langle In_N(I_N)\rangle}=\sqrt{\I_{\G,x}}$. 
The left term inclusion is immediate since $\langle In_N(I_N)\rangle\subset\I_{\G,x}$. By Lemma \ref{lemma_aux1}, it is enough to check the other inclusion for the generators of $\G$. As $f_{n_i}^{\alpha_i}\in I_N$ for $\alpha_i=\frac{N}{n_i}$, then $In_{n_i}(f_{n_i})^{\alpha_i}\in In_N(I_N)$ and the equality holds.
\end{itemize}
\end{proof}

\begin{lemma}\label{lemma_aux3}
Set $N$ and $\G_N$ as in Lemma \ref{lemma_aux2}. Then, $\tau_{\G,x}=\tau_{\G_N,x}$.
\end{lemma}

\begin{proof}
By definition, $\tau_{\G,x}=\codim(\L_{\G,x})$.  $\L_{\G,x}$ is the linear space of vertices of the zeroset of $\I_{\G,x}$, so it is enough to consider the zeroset of $\sqrt{\I_{\G,x}}$ to define $\L_{\G,x}$. By Lemma \ref{lemma_aux2}, $\sqrt{\I_{\G,x}}=\sqrt{\I_{\G_N,x}}$ so $V(\sqrt{\I_{\G,x}})$ and $V(\sqrt{\I_{\G_N,x}})$ have the same subspace of vertices, that is, $\L_{\G,x}=\L_{\G_N,x}$. So, finally, $\tau_{\G,x}=\tau_{\G_N,x}$.

\end{proof}

\begin{proof} [Proof of Theorem \emph{\ref{equal_tau}}]
Assume that locally at $x$, that the Rees algebras are defined by $\G=\oplus I_nW^n=\calo_{V}[f_{1}W^{n_1},\dots,f_{r}W^{n_r}]$ and $\G'=\oplus I'_nW^n=\calo_{V}[g_{1}W^{m_1},\dots,g_{s}W^{m_s}]$.  Choose a positive integer $N$, such that $N$ is a common multiple of every $n_i$ and every $m_j$, for $i=1,\dots,r$ and $j=1,\dots,s$. Consider the Rees rings attached to $\G$ and $\G'$ given by $\G_N=\calo_{V}[I_NW^N]$ and $\G'_N=\calo_{V}[I'_NW^N]$, respectively. 

It suffices to consider the case where $\G\subset \G'$, so then $I_N\subset I'_N$ and this inclusion is an integral extension of ideals. It follows that $\langle In_N(I_N)\rangle\subset\langle In_N(I'_N)\rangle$ is an integral extension of ideals, as one can check that the conditions of integral dependence hold for the generators. Then, $\sqrt{\langle In_N(I_N)\rangle}=\sqrt{\langle In_N(I'_N)\rangle}$ and by Lemmas \ref{lemma_aux2} and \ref{lemma_aux3} we conclude  that $\tau_{\G,x}=\tau_{\G_N,x}=\tau_{\G'_N,x}=\tau_{\G',x}$, which proves the Proposition.
\end{proof}

\section{The $\tau$-invariant and the elimination algebra}\label{sect:tau:elim}

\begin{parrafo}
In this section we present a result relating the $\tau$-invariant of a differential Rees algebra $\G$ with the $\tau$-invariant of the elimination algebra attached to $\G$. In order to prove this result, we show before a local presentation of $\G$ in terms of the elimination algebra. Although this local presentation is very important by itself  and has other applications (see also \cite{BeV}), we use it here to study the relationship between both $\tau$-invariants.

A weaker relation between the $\tau$-invariants of a differential Rees algebra and its elimination Rees algebra appears in  \cite{VV4}, Proposition 5.12. 
\end{parrafo}

\begin{parrafo}Recall that given a Rees algebra $\G$ and a closed point $x\in \Sing(\G)$, then a tangent cone, say $\C_{\mathcal{G}}$, at $\mathbb{T}_{V^{(d)},x}=\Spec(gr_{\M_x}( {\calo}_{V^{(d)},x}))$ is defined by an homogeneous ideal 
$\I_{\G,x}$ in $gr_{\M_x}( {\calo}_{V^{(d)},x})$ (see \ref{vertices}). If  $\G=\calo_V[f_{1}W^{n_1},\dots,f_{s}W^{n_s}]$, locally at $x$, then $\I_{\G,x}=\langle In_{n_1}(f_{1}),\dots,In_{n_s}(f_{s})\rangle$ (see Lemma \ref{lemma_aux1}). It was indicated in \ref{vertices} that there is a largest  subspace, say  $\L_{\mathcal{G}}$,  included and acting by translations on $\C_{\mathcal{G}}$, and  $\tau_{\G,x}$ is defined as the codimension of this linear subspace. Furthermore,  $\tau_{\G,x}\geq 1$ whenever  $\I_{\G,x}$ is non-zero.
A projection $V^{(d)}\overset{\beta}{\longrightarrow}V^{(d-1)}$ is said to be transversal to $\G$ at $x$ if the tangent line of the fiber $\beta^{-1}(\beta(x))$ at $x$ is not included in the subspace $\L_{\mathcal{G}}$.

Fix a regular system of parameters $\{y_1,y_2,\dots ,y_{d-1}\}$ at $\calo_{V^{(d-1)},\beta(x)}$, and choose an element $Z$ so that $\{y_1,y_2,\dots ,y_{d-1},Z\}$ is a regular system of parameters at $\calo_{V^{(d)},x}$.
Recall that Rees algebras are to be considered up to integral closure. So if the condition of transversality holds one can modify the local generators  $\{f_{1}W^{n_1},\dots,f_{s}W^{n_s}\}$ of $\G$ so that   each
\begin{equation}\label{req1}
f_{i}=Z^{n_i}+a_1^{(i)}Z^{n_i-1}+\cdots + a^{(i)}_{n_i}\in \calo_{V^{(d-1)},\beta(x)}[Z]
\end{equation}
is a monic polynomial in $Z$.

Assume now that $\G$ is a differential Rees algebra relative to $V^{(d)}\overset{\beta}{\longrightarrow}V^{(d-1)}$. In such case $\G$ can be identified, locally at $x$, by
$$\calo_{V^{(d-1)},x}[Z][f_{n_i}(Z)W^{n_i},\Delta^{(e_i)}(f_{n_i}(Z))W^{n_i-e_i}]_{1\leq e_i\leq n_i-1,\ 1\leq i\leq s},$$
via the natural inclusion $\calo_{V^{(d-1)},x}[Z]\subset \calo_{V^{(d)},x}$, which is a specialization of a the universal elimination algebra defined for $s$ monic polynomials as in $(\ref{eqara1})$. In particular an elimination algebra, say $\R_{\G,\beta}\subset \calo_{V^{(d-1)},\beta(x)}[W]$ is defined as in $(\ref{eqara2})$. 
\end{parrafo}

The following Proposition shows a local presentation of a differential Rees algebra in terms of a monic polynomial and the elimination algebra. See also \cite{BeV}.

\begin{proposition}\label{pr_local}{\bf (Local relative presentation)}. 
Let  $x\in\Sing(\G)$ be a close point such that $\tau_{\G,x}\geq 1$. Consider a projection  $V^{(d)}\overset{\beta}{\longrightarrow} V^{(d-1)}$ which is transversal at $x$. Assume that $\G$ is a $\beta$-relative differential algebra and that there exists an element $f_nW^n\in\G$ such that $f_n$ has order exactly $n$ at the local regular ring $\calo_{V^{(d)},x}$, and $f_n=f_n(Z)$ is a monic polynomial of degree $n$ in $\calo_{V^{(d-1)},\beta(x)}[Z]$. Then, at a suitable neighborhood of $x$:
$$\G\sim\calo_{V^{(d)}}[f_n(Z)W^n,\De(f_n(Z))W^{n-e}]_{1\leq e\leq n-1}\odot\R_{\G,\beta},$$
where $\R_{\G,\beta}$ in $\calo_{V^{(d)}}[W]$ is identified with $\beta^*(\R_{\G,\beta})$ and the equivalence $\sim$ is that in Definition $\ref{equiv_Rees}$.
\end{proposition}

\begin{proof}
We may assume that $f_n(Z)\in \{f_{1}W^{n_1},\dots,f_{s}W^{n_s}\}$ as in (\ref{req1}).
Let us check these assertions in the universal case. In order to simplify notation we consider here the case of two generators (i.e., the case $s=2$). So consider variables $Z$, $Y_i$ y $V_j$ over a field $k$, and
$$F_n(Z)=(Z-Y_1)\cdot(Z-Y_2)\dots(Z-Y_n).$$
This polynomial is universal of degree $n$, and  $f_n=f_n(Z)$ is a pull-back of  $F_n(Z)$. Let
$$G_m(Z)=(Z-V_1)\cdot(Z-V_2)\dots(Z-V_m)$$
be the universal polynomial of degree $m$. A natural inclusion $\R_{\G,\beta}\subset \G$ arises from 
$(\ref{eqrp})$.

The permutation group $\S_n\times \S_m$ acts on $k[Z,Y_1,\dots,Y_n,V_1,\dots,V_m]$. This group also acts on 
$$S=k[Z-Y_1,Z-Y_2,\dots,Z-Y_n,Z-V_1,Z-V_2,\dots,Z-V_m].$$
The subring of invariants of $S$, $S^{\S_n\times\S_m}$, is
$$k[\De(F_n(Z)),\Delta^{(e')}(G_m(Z))]_{0\leq e\leq n-1,\ 0\leq e'\leq m-1},$$
where $\De(F_n(Z))$ is an homogeneous polynomial of degree $n-e$ and $\Delta^{(e')}(G_m(Z))$ is homogeneous of degree $m-e'$. We add a dummy variable $W$ that simply will express the degree. Then, the subring of invariants $S^{\S_n\times\S_m}$ is
$$k[\De(F_n(Z))W^{n-e},\Delta^{(e')}(G_m(Z))W^{m-e'}]_{0\leq e\leq n-1,\ 0\leq e'\leq m-1}.$$

The universal elimination algebra is defined as the invariant ring of $\S_n\times\S_m$ acting on
$$ S'=k[(Y_2-Y_1),\dots,(Y_n-Y_1),(V_1-Y_1),\dots,(V_m-Y_1)].$$

The key observation to prove the assertion is that  $S$ is spanned by two subrings: $k[Z-Y_1,\dots,Z-Y_n]$ and $S'$. The subring of invariants on the first is 
$k[\De(F_n(Z))W^{n-e}]_{0\leq e\leq n-1}$
and the one of the second is the universal elimination algebra.  

So, both invariant algebras are included in $S^{\S_n\times\S_m}$; and in order to prove the claim it suffices to show that $S^{\S_n\times\S_m}$ is an integral extension of the subalgebra spanned by the two invariant subalgebras. To prove this last condition note that $S$ is an integral extension of the subalgebra spanned by the two invariant subalgebras. This proves the claim.
\end{proof}

\begin{theorem}\label{tauG_elim}
Let $V^{(d)}$ be a smooth scheme of dimension $d$, $\G$ a differential Rees algebra, $x\in\Sing(\G)$ a closed point, and suppose that $\tau_{\G,x}\geq 1$. Fix a generic projection $V^{(d)}\overset{\beta}{\longrightarrow}V^{(d-1)}$ (see \emph{\ref{genproj}}). Then an elimination algebra $\R_{\G,\beta}$ is defined locally at $\calo_{V^{(d-1)},\beta(x)}$ , and its $\tau$-invariant drops by one, that is,
$$\tau_{\R_{\G,\beta},\beta(x)}=\tau_{\G,x}-1.$$
\end{theorem}

\begin{proof}
Fix a regular system of parameters $\{x_1,\dots,x_d\}$ of $\calo_{V^{(d)},x}$. So the graded ring is given by $k'[X_1,\dots,X_d]$ where $X_i$ denotes the initial form of $x_i$. Here $k'$ is the residue field of the local ring at the closed point. If we assume that $k'$ is a perfect field,for a suitable choice of $\{x_1,\dots,x_{d}\}$, then
$\I_{\G,x}=\langle X_1^{p^{e_1}},\dots, X_r^{p^{e_r}}\rangle$ for certains non-negative integers $e_i,i=1,\dots,r$ (see Remark \ref{rmk:initial_forms}). 

So, there exist elements $f_iW^{p^{e_i}}\in\G$ for $i=1,\dots,r$, such that $In_{p^{e_i}}(f_i)=X_i^{p^{e_i}}$. Since $\beta$ is generic, then $\C_\G\cap\ell=\{0\}$, where $\ell$ denotes the tangent line to the fiber of the projection. It follows that $\C_{f_i}\cap\ell=\{0\}$ for some $i\in\{1,\dots,r\}$. Assume that this condition is achieved by $i=1$. 

By Weierstrass Preparation Theorem it follows that in an \'etale neighborhood of $x$, $f_1(x_1)=x_1^{p^{e_1}}+a_1x_1^{p^{e_1}-1}+\dots+a_{p^{e_1}}$, where $a_i\in\calo_{V^{(d-1)},\beta(x)}$ and the order of each $a_i$ is  $>i$. Note that we obtain a strict inequality since $In_{p^{e_1}}(f_1)=X_1^{p^{e_1}}$.

By \ref{pr_local} we may assume that there is a local relative presentation of the form
$$\G\sim \calo_{V^{(d)}}[f_1(x_1)W^{p^{e_1}},\Delta^{(\alpha)}(f_1(x_1))W^{p^{e_1}-\alpha}]_{1\leq \alpha \leq p^{e_1}-1}\odot\beta^*(\R_{\G,\beta}).$$

Denote by $\widetilde{\G}$ the Rees algebra defined as 
$$\widetilde{\G}=\calo_{V^{(d)}}[f_1(x_1)W^{p^{e_1}},\Delta^{(\alpha)}(f_1(x_1))W^{p^{e_1}-\alpha}]_{1\leq \alpha \leq p^{e_1}-1}.$$ 
Since $In_{p^{e_1}}(f(x_1)=X_1^{p^{e_1}}$, then $\tau_{\widetilde{\G},x}=1$ and the only variable needed to define the generators of  $\I_{\widetilde{\G},x}$ is $X_1$. On the other hand, we have eliminated the variable $x_1$ at $\beta^*(\R_{\G,\beta})$, so $X_1$ is not needed to define generators of $\I_{\beta^*(\R_{\G,\beta})}$.

Finally,  since $\G\sim\widetilde{\G}\odot\beta^*(\R_{\G,\beta})$ applying Theorem \ref{equal_tau}, it follows that $\tau_{\G}=\tau_{\G_1\odot\R_{\G,\beta}}$, and by the previous arguments, $\tau_{\G_1\odot\R_{\G,\beta}}=1+\tau_{\R_{\G,\beta}}$. Finally,
$\tau_{\G}=1+\tau_{\R_{\G,\beta}}.$
\end{proof}

In the following corollary, we present the relationship  between $\tau$-invariants in a more general setting; whenever $\G$ is only a relative differential algebra (see \cite{VV4}, \cite{BV3} or \cite{BeV} for more applications of this relative differential algebras in the problem of resolution of singularities). 

\begin{corollary}
In the general case, when $\G$ is relative differential to $V^{(d)}\overset{\beta}{\longrightarrow} V^{(d-1)}$, the following relationship between the $\tau$-invariants holds
$$\tau_{\R_{\G,\beta}}\leq \tau_{\G}-1$$
\end{corollary}

\begin{proof}
Since, $\G\subset G(\G)$, where $G$ denotes the Giraud's operator (see Theorem \ref{thopG}), it follows that $\R_{\G,\beta}\subset \R_{G(\G),\beta}=\mathcal{H}$, where $\R_{G(\G),\beta}$ denotes the elimination algebra attached to $G(\G)$. By the previous inclusion and Theorem \ref{tauG_elim},
$$\tau_{\R_{\G,\beta}}\leq \tau_{\mathcal{H}}=\tau_{G(\G)}-1=\tau_{\G}-1,$$
where the last equality holds because of the compatibility of the $\tau$-invariant with the differential operator.
\end{proof}

\end{document}